\documentclass{amsart}
\usepackage{amssymb}
\usepackage{amsmath}
\usepackage{amsthm}
\usepackage{graphicx}
\usepackage{float}
\usepackage{color}
\usepackage{csquotes}

\newcommand{\re}{\mathop{\mathrm{Re}}}
\newcommand{\im}{\mathop{\mathrm{Im}}}

\newcommand{\artanh}{\mathop{\mathrm{artanh}}}

\newtheorem{theorem}{Theorem}
\newtheorem{lemma}{Lemma}

\newtheorem{remark}{Remark}

\newtheorem{corollary}{Corollary}

\title[]{A note on the Schwarz lemma for harmonic functions}

\author{Marek Svetlik}
\address{Faculty of mathematics, University of Belgrade, Studentski Trg 16,
Belgrade, Republic of Serbia} \email{svetlik@matf.bg.ac.rs}

\date{November 8, 2019}

\keywords{The Schwarz lemma; the Schwarz-Pick lemma; harmonic functions}

\subjclass[2010]{Primary 30C80; Secondary 31C05, 30C75.}

\thanks{Research partially supported by Ministry of Education, Science and Technological Development
of the Republic of Serbia, Grant No. 174 032.}

\begin{document}

\maketitle

\begin{abstract}
In this note we  consider some generalizations of the Schwarz lemma for harmonic functions on the unit disk, whereby values of such functions and  the norms of their differentials at the point $z=0$ are given.
\end{abstract}

\section{Introduction}
%

\subsection{A summary of some results}

In this paper we consider some generalizations of the Schwarz lemma for harmonic functions from the unit disk $\mathbb{U}=\{z\in\mathbb{C}:|z|<1\}$ to the interval $(-1,1)$ (or to itself).

First, we cite a theorem which is known as the Schwarz lemma for harmonic functions and is considered as a classical result.

\begin{theorem}[\cite{heinz},{\cite[p.77]{duren}}]\label{th:schwhar}
Let $f:\mathbb{U}\rightarrow\mathbb{U}$ be a harmonic function such that $f(0)=0$. Then
\begin{equation*}
    |f(z)|\leqslant\frac{4}{\pi}\arctan{|z|}, \quad \mbox{ for all } \quad z\in\mathbb{U},
\end{equation*}
and this inequality is sharp for each point $z\in\mathbb{U}$.
\end{theorem}

In 1977, H. W. Hethcote \cite{heth} improved this result by removing the assumption $f(0)=0$ and proved the following theorem.

\begin{theorem}[{\cite[Theorem 1]{heth}} and {\cite[Theorem 3.6.1]{pavlovic}}]\label{th:schwheth}
Let $f:\mathbb{U}\rightarrow\mathbb{U}$ be a harmonic function. Then
\begin{equation*}
    \left|f(z) -\frac{1-|z|^2}{1 + |z|^2} f(0)\right|\leqslant \frac{4}{\pi}\arctan|z|, \quad \mbox{ for all } \quad z\in\mathbb{U}.
\end{equation*}
\end{theorem}

As it was written in \cite{MMandMS}, it seems that researchers have had some difficulties to handle the case $f(0)\neq0$, where $f$ is harmonic mapping from $\mathbb{U}$ to itself. Before explaining the essence of these difficulties, it is necessary to recall one mapping and some of its properties. Also, emphasize that this mapping and its properties have an important role in our results.

Let $\alpha\in\mathbb{U}$ be arbitrary. Then for $z\in\mathbb{U}$ we define $\displaystyle\varphi_{\alpha}(z)=\frac{\alpha+z}{1+\overline{\alpha}z}$. It is well known that $\varphi_{\alpha}$ is a conformal automorphism of $\mathbb{U}$. Also, for $\alpha\in(-1,1)$ we have
\begin{itemize}
  \item[$1^\circ$] $\varphi_{\alpha}$ is increasing on $(-1,1)$ and maps $(-1,1)$ onto itself;
  \item[$2^\circ$] $\displaystyle\varphi_{\alpha}([-r,r])=[\varphi_{\alpha}(-r),\varphi_{\alpha}(r)]=\left[\frac{\alpha-r}{1-\alpha r},\frac{\alpha+r}{1+\alpha r}\right]$, where $r\in[0,1)$.
\end{itemize}

Now we can explain the mentioned difficulties. If $f$ is holomorphic mapping from $\mathbb{U}$ to $\mathbb{U}$, such that $f(0)=b$,  then using the mapping $g=\varphi_{-b}\circ f$ we can reduce described problem to the case $f(0)=0$. But, if  $f$ is harmonic mapping from $\mathbb{U}$ to $\mathbb{U}$ such that $f(0)=b$, then the mapping $g=\varphi_{-b}\circ f$ doesn't have to be harmonic mapping.

In joint work \cite{MMandMS} of the author with M. Mateljevi\'c, the Theorem \ref{th:schwhar} was proved in a different way than those that could be found in the literature (for example, see \cite{heinz, duren}). Modifying that proof in an obvious way, the following theorem (which can be considered as an improvement of the H. W. Hethcote result) has also been proved in \cite{MMandMS}.




\begin{theorem}[{\cite[Theorem 6]{MMandMS}}]\label{th:schwhar1}
Let $u:\mathbb{U}\rightarrow(-1,1)$ be a harmonic function such that $u(0)=b$. Then
\begin{equation*}
\frac{4}{\pi}\arctan\varphi_a(-|z|)\leqslant u(z)\leqslant  \frac{4}{\pi}\arctan\varphi_a(|z|), \quad \mbox{for all} \quad z\in\mathbb{U}.
\end{equation*}
Here $\displaystyle a=\tan{\frac{b\pi}{4}}$. Also, these inequalities are both sharp at each point $z\in\mathbb{U}$.
\end{theorem}

As one corollary of Theorem \ref{th:schwhar1} it is possible to prove the following theorem.

\begin{theorem}[{\cite[Theorem 1]{MMandAK}}]\label{th:schwar1complex}
Let $f:\mathbb{U}\rightarrow\mathbb{U}$ be a harmonic function such that $f(0)=b$. Then
\begin{equation*}
    |f(z)|\leqslant\frac{4}{\pi}\arctan\varphi_{A}(|z|), \quad \mbox{for all} \quad z\in\mathbb{U}.
\end{equation*}
Here $\displaystyle A=\tan{\frac{|b|\pi}{4}}$.

\end{theorem}

This paper gives a relatively elementary contribution and continuation to the mentioned approach. We give further generalizations of Theorems \ref{th:schwhar1} and \ref{th:schwar1complex}. These generalizations (see Theorems \ref{th:main1} and \ref{th:main2}) consist of considering harmonic functions on the unit disk $\mathbb{U}$ with following additional conditions:
\begin{itemize}
  \item[1)] the value at the point $z=0$ is given; 
  \item[2)] the values of partial derivatives at the point $z=0$ are given. 
\end{itemize}

In the literature one can find the following two generalizations of the Schwarz lemma for holomorphic functions.

\begin{theorem}[{\cite[Proposition 2.2.2 (p. 32)]{krantz}}]\label{th:lindelof}
Let $f:\mathbb{U}\rightarrow\mathbb{U}$ be a holomorphic function. Then
\begin{equation}\label{fla1:krantz}
    |f(z)|\leqslant\frac{|f(0)|+|z|}{1+|f(0)||z|},\quad \mbox{for all} \quad z\in\mathbb{U}.
\end{equation}
\end{theorem}

\begin{theorem}[{\cite[Proposition 2.6.3 (p. 60)]{krantz}}, {\cite[Lemma 2]{osserman}}]\label{th:osserman}
Let $f:\mathbb{U}\rightarrow\mathbb{U}$ be a holomorphic function such that $f(0)=0$. Then
\begin{equation}\label{fla1:krantz}
    |f(z)|\leqslant|z|\frac{|f'(0)|+|z|}{1+|f'(0)||z|},\quad \mbox{for all} \quad z\in\mathbb{U}.
\end{equation}
\end{theorem}

S. G. Krantz in his book \cite{krantz} attribute Theorem~\ref{th:lindelof}  to Lindel\"of. Note that Theorem~\ref{th:schwar1complex} could be considered as harmonic version of Theorem \ref{th:lindelof}. Similarly, one of the main result of this paper (Theorem \ref{th:main2}) could be considered as harmonic version of Theorem \ref{th:osserman}.

\subsection{Hyperbolic metric and Schwarz-Pick type estimates}\label{subsec:hyperbolic}
By $\Omega$ we denote a simply connected plane domain different from $\mathbb{C}$ (we call these domains hyperbolic). By Riemann's Mapping Theorem, it follows that any such domain are conformally equivalent to the unit disk $\mathbb{U}$. Also, a domain $\Omega$ is equipped with the hyperbolic metric $\rho_{\Omega}(z)|dz|$. More precisely, by definition we have
$$\displaystyle\rho_{\mathbb{U}}(z)=\frac{2}{1-|z|^2}$$
and if $f:\Omega\rightarrow\mathbb{U}$ a conformal isomorphism, then also by definition, we have $$\rho_{\Omega}(w)=\rho_{\mathbb{U}}(f(w))|f'(w)|.$$
The hyperbolic metric induces a hyperbolic distance on $\Omega$ in the following way
\begin{equation*}
    d_{\Omega}(z_1,z_2)=\inf\int_{\gamma}\rho_{\Omega}(z)|dz|,
\end{equation*}
where the infimum is taken over all $C^1$ curves $\gamma$ joining $z_1$ to $z_2$ in $\Omega$. For example, one can show that
\begin{equation*}
    d_{\mathbb{U}}(z_1,z_2)=2\artanh{\left|\frac{z_1-z_2}{1-z_1\overline{z_2}}\right|},
\end{equation*}
where $z_1,z_2\in\mathbb{U}$.


Hyperbolic metric and hyperbolic distance do not increase under a holomorphic function. More precisely, the following well-known theorem holds.

\begin{theorem}[The Schwarz-Pick lemma for simply connected domains, {\cite[Theorem~6.4.]{BeardonMinda}}]\label{th:schwpick}
Let $\Omega_1$ and $\Omega_2$ be hyperbolic domains and $f:\Omega_1\rightarrow\Omega_2$ be a holomorphic function. Then
\begin{equation}\label{schwpick:fla1}
    \rho_{\Omega_2}(f(z))|f'(z)|\leqslant\rho_{\Omega_1}(z), \quad \mbox{ for all } \quad z\in\Omega_1,
\end{equation}
and
\begin{equation}\label{schwpick:fla2}
    d_{\Omega_2}(f(z_1),f(z_2))\leqslant d_{\Omega_1}(z_1,z_2), \quad \mbox{ for all } \quad z_1,z_2\in\Omega_1.
\end{equation}
If $f$ is a conformal isomorphism from $\Omega_1$ onto $\Omega_2$ then in (\ref{schwpick:fla1}) and (\ref{schwpick:fla2}) equalities hold. On other hand if either equality holds in (\ref{schwpick:fla1}) at one
point $z$ or for a pair of distinct points in (\ref{schwpick:fla2}) then $f$ is a conformal isomorphism from $\Omega_1$ onto $\Omega_2$.
\end{theorem}

%
%
%

\
For holomorphic function $f:\Omega_1\rightarrow\Omega_2$ (where $\Omega_1$ and $\Omega_2$ are hyperbolic domains) it's defined (for motivation and details see Section 5 in \cite{BeardonMinda}, cf. \cite{BeardonCarne}) the \emph{hyperbolic distortion} of $f$ at $z\in\Omega_1$ on the following way $$\displaystyle|f^h(z)|=\frac{\rho_{\Omega_2}(f(z))}{\rho_{\Omega_1}(z)}|f'(z)|.$$

Note that by Theorem \ref{th:schwpick} we also have $|f^h(z)|\leqslant 1$ for all $z\in\Omega_1$.

Using this notion, in 1992, A. F. Beardon and T. K. Carne proved the following theorem which is stronger than Theorem \ref{th:schwpick}.

\begin{theorem}[\cite{BeardonCarne}]\label{th:BeardonCarne}
Let $\Omega_1$ and $\Omega_2$ be hyperbolic domains and $f:\Omega_1\rightarrow\Omega_2$ be a holomorphic function. Then for all $z,w\in\Omega_1$,
\begin{equation*}
    d_{\Omega_2}(f(z),f(w))\leqslant\log(\cosh{d_{\Omega_1}(z,w)}+|f^h(w)|\sinh{d_{\Omega_1}(z,w)}).
\end{equation*}
\end{theorem}
Let us note that Theorem \ref{th:BeardonCarne} is of crucial importance for our research (see proof of Theorem \ref{th:main1}).

There are many papers where authors have considered various versions of Schwarz-Pick type estimates for harmonic functions (see \cite{khavinson, burgeth, MKandMM, colonna, kavu, hhChen, MMar}). In this regard, we note that M. Mateljevi\'c \cite{MMSchw_Kob} recently explained one method (refer to it as the strip method) which enabled that some of these results to be proven in an elegant way.


For completeness we will shortly reproduce the strip method. In order to do so, we will first introduce the appropriate notation and specify some simple facts.

By $\mathbb{S}$ we denote the strip $\{z\in\mathbb{C}:-1<\re{z}<1\}$. The mapping $\varphi$ defined by $\displaystyle\varphi(z)=\tan{\left(\frac{\pi}{4}z\right)}$ is a conformal isomorphism from $\mathbb{S}$ onto $\mathbb{U}$ and by $\phi$ we denote the inverse mapping of $\varphi$ (see also Example 1 in \cite{MMandMS}). Throughout this paper by $\varphi$ and $\phi$ we always denote these mappings.

Using the mapping $\varphi$ one can derive the following equality
$$
\rho_{\mathbb{S}}(z)=\rho_{\mathbb{U}}(\varphi(z))|\varphi'(z)|=\frac{\pi}{2}\frac{1}{\cos{\left(\displaystyle\frac{\pi}{2}\re{z}\right)}}, \quad \mbox{ for all } \quad z\in\mathbb{S}.
$$

By $\nabla u$ we denote the gradient of real-valued $C^1$ function $u$, i.e. $\nabla u=(u_x,u_y)=u_x+iu_y$. If $f=u+iv$ is complex-valued $C^1$ function, where $u=\re{f}$ and $v=\im{f}$, then we use notation
$$f_{x}=u_{x}+iv_{x} \qquad \mbox{ and } \qquad f_{y}=u_{y}+iv_{y},$$
as well as
$$f_{z}=\displaystyle\frac{1}{2}(f_{x}-if_{y}) \qquad \mbox{ and } \qquad f_{\overline{z}}=\displaystyle\frac{1}{2}(f_{x}+if_{y}).$$
Finally, by $df(z)$ we denote differential of the function $f$ at point $z$, i.e. the Jacobian matrix
\begin{equation*}
\left(
  \begin{array}{cc}
    u_x(z) & u_y(z) \\
    v_x(z) & v_y(z) \\
  \end{array}
\right).
\end{equation*}

The matrix $df(z)$ is an $\mathbb{R}$-linear operator from the tangent $T_{z}\mathbb{R}^2$ to the tangent space $T_{f(z)}\mathbb{R}^2$. By $\|df(z)\|$ we denote norm of this operator. It is not difficult to prove that $\|df(z)\|=|f_z(z)|+|f_{\overline{z}}(z)|$.

Briefly, the strip method consists of the following elementary considerations (see \cite{MMandMS}):

\begin{itemize}
\item[(I)] Suppose that $f:\mathbb{U}\rightarrow\mathbb{S}$ be a holomorphic function. Then by Theorem \ref{th:schwpick} we have $\rho_{\mathbb{S}}(f(z))|f'(z)|\leqslant\rho_{\mathbb{U}}(z)$, for all $z\in \mathbb{U}$.
\item[(II)] If $f=u+iv$ is a harmonic function and $F=U+iV$ is a holomorphic function on a domain  $D$   such that $\re{f}=\re{F}$  on  $D$ (in this setting we say that $F$ is associated to $f$ or to $u$), then $F'=U_x+iV_x = U_x-iU_y = u_x-iu_y$. Hence $F'=\overline{\nabla u}$ and $|F'|=|\overline{\nabla u}|=|\nabla u|$.
\item[(III)] Suppose that $D$ is a simply connected plane domain and $f:D \rightarrow \mathbb{S}$ is a harmonic function. Then it is known from the standard course of complex analysis that there is a holomorphic function $F$ on $D$  such that  $\re{f} = \re{F}$ on $D$, and it is clear that $F:D\rightarrow\mathbb{S}$.
\item[(IV)] The hyperbolic density $\rho_{\mathbb{S}}$ at point $z$  depends only on  $\re{z}$.
\end{itemize}
By  (I)-(IV)  it is readable  that we have the following theorem.
\begin{theorem}[{\cite[Proposition 2.4]{MMSchw_Kob}}, \cite{kavu,hhChen}]\label{th:schwpickhar}
Let $u:\mathbb{U}\rightarrow(-1,1)$ be a harmonic function and let $F$ be a holomorphic function which  is associated to $u$. Then
\begin{equation}\label{schwpickhar:fla1}
    \rho_{\mathbb{S}}(u(z))|\nabla u(z)|=\rho_{\mathbb{S}}(F(z))|F'(z)|\leqslant \rho_{\mathbb{U}}(z),\quad \mbox{ for all }\quad z\in\mathbb{U}.
\end{equation}
In other words
\begin{equation}\label{schwpickhar:fla2}
|\nabla u(z)|\leqslant\frac{4}{\pi}\frac{\displaystyle\cos{\left(\frac{\pi}{2}u(z)\right)}}{1-|z|^2}, \quad \mbox{ for all }\quad z\in\mathbb{U}.
\end{equation}
If $u$ is real part of a conformal isomorphism from $\mathbb{U}$ onto $\mathbb{S}$ then in (\ref{schwpickhar:fla2}) equality holds for all $z\in\mathbb{U}$ and vice versa.
\end{theorem}

In 1989, F. Colonna \cite{colonna} proved the following version of the Schwarz-Pick lemma
for harmonic functions.

\begin{theorem}[{\cite[Theorem 3]{colonna}} and {\cite[Proposition 2.8]{MMSchw_Kob}}, cf. {\cite[Theorem 6.26]{axler}}]\label{th:colonna}
Let $f:\mathbb{U}\rightarrow\mathbb{U}$ be a harmonic function. Then
\begin{equation}\label{colonna:fla1}
    \|df(z)\|\leqslant\frac{4}{\pi}\frac{1}{1-|z|^2}, \quad \mbox{ for all } z\in\mathbb{U}.
\end{equation}
In particular,
\begin{equation}\label{colonna:fla2}
    \|df(0)\|\leqslant\frac{4}{\pi}.
\end{equation}
\end{theorem}

\begin{remark} The inequality (\ref{colonna:fla1}) is sharp in the following sense: for all $z\in\mathbb{U}$ there exists a harmonic function $f^z:\mathbb{U}\rightarrow\mathbb{U}$ (which depends on $z$) such that
\begin{equation*}
     \|df^z(z)\|=\frac{4}{\pi}\frac{1}{1-|z|^2}.
\end{equation*}
One such function is defined by $f(\zeta)=\re{(\phi(\varphi_{-z}(\zeta)))}$. For more details see Theorem 4 in \cite{colonna}.
\end{remark}

\begin{remark}
The inequality (\ref{colonna:fla2}) could not be improved even if we add the assumption that $f(0)=0$. More precisely, if $f(\zeta)=\re\phi(\zeta)$ then $f$ satisfy all assumptions of Theorem \ref{th:colonna}, $f(0)=0$ and $\displaystyle\|df(0)\|=\frac{4}{\pi}$
(see also {\cite[Proposition 2.8]{MMSchw_Kob}} and {\cite[Theorem 6.26]{axler}}).\end{remark}

\begin{remark}\label{rm:problem}
It seems that question: ,,Is it possible to improve the inequality (\ref{colonna:fla1}), if we add the assumption $f(0)=b$, where $b\neq0$?'' is an open problem (see {\cite[Problem 2]{MMSchw_Kob}}).
\end{remark}


Note that the inequalities (\ref{schwpickhar:fla2}) and (\ref{colonna:fla2}) naturally impose assumptions in the Theorems \ref{th:main1} and \ref{th:main2} below.

\section{Main results}

\begin{theorem}\label{th:main1}
Let $u:\mathbb{U}\rightarrow(-1,1)$ be a harmonic function such that:
\begin{itemize}
  \item[(R1)] $u(0)=b$ and
  \item[(R2)] $|\nabla u(0)|=d$, where $\displaystyle d\leqslant\frac{4}{\pi}\cos{\left(\frac{\pi}{2}b\right)}$.
\end{itemize}
Then, for all $z\in\mathbb{U}$,
\begin{equation}\label{ineq:main}
\frac{4}{\pi}\arctan{\varphi_{a}\big(-|z|\varphi_{c}(|z|)\big)}\leqslant u(z)\leqslant\frac{4}{\pi}\arctan{\varphi_{a}\big(|z|\varphi_{c}(|z|)\big)}.
\end{equation}
Here $\displaystyle a=\tan\frac{b\pi}{4}$ and $c=\displaystyle\frac{\pi}{4}\frac{1}{\cos{\displaystyle\frac{\pi}{2}b}}d$.
These inequalities are sharp for each point $z\in\mathbb{U}$ in the following sense: for arbitrary $z\in\mathbb{U}$ there exist harmonic functions $\widehat{u}^z, \widetilde{u}^z:\mathbb{U}\rightarrow(-1,1)$, which depend on $z$, such that they satisfy {\rm(R1)} and {\rm(R2)} and also
\begin{equation*}
    \widehat{u}^z(z)=\frac{4}{\pi}\arctan{\varphi_{a}\big(-|z|\varphi_{c}(|z|)\big)}\quad \mbox{and} \quad \widetilde{u}^z(z)=\frac{4}{\pi}\arctan{\varphi_{a}\big(|z|\varphi_{c}(|z|)\big)}.
\end{equation*}
\end{theorem}

\begin{remark}
Formally, if $c=1$ then function $\varphi_c$ is not defined. In this case we mean that $\varphi_{c}(|z|)=1$ for all $z\in\mathbb{U}$.
\end{remark}

\begin{corollary}
Let $u:\mathbb{U}\rightarrow(-1,1)$ be a harmonic function such that $u(0)=0$ and $\nabla u(0)=(0,0)$. Then, for all $z\in\mathbb{U}$,
\begin{equation*}
    |u(z)|\leqslant\frac{4}{\pi}\arctan{|z|^2}.
\end{equation*}
\end{corollary}

\begin{theorem}\label{th:main2}
Let $f:\mathbb{U}\rightarrow\mathbb{U}$ be a harmonic function such that:
\begin{itemize}
  \item[(C1)] $f(0)=0$ and
  \item[(C2)] $\|df(0)\|=d$, where $\displaystyle d\leqslant\frac{4}{\pi}$.
\end{itemize}
Then, for all $z\in\mathbb{U}$
\begin{equation}\label{main2:ineq}
    |f(z)|\leqslant\frac{4}{\pi}\arctan\big(|z|\varphi_{C}(|z|)\big),
\end{equation}
where $\displaystyle C=\frac{\pi}{4}d$.
\end{theorem}

\begin{corollary}
Let $f:\mathbb{U}\rightarrow\mathbb{U}$ be a harmonic function such that $f(0)=0$ and $\|df(0)\|=0$. Then, for all $z\in\mathbb{U}$,
\begin{equation*}
    |f(z)|\leqslant\frac{4}{\pi}\arctan{|z|^2}.
\end{equation*}
\end{corollary}

\begin{remark}
Formally, if $C=1$ then function $\varphi_C$ is not defined. In this case we mean that $\varphi_{C}(|z|)=1$ for all $z\in\mathbb{U}$.
\end{remark}

\section{Proofs of main results}

\subsection{Proof of Theorem \ref{th:main1}} In order to prove Theorem \ref{th:main1}, we recall the following definitions and one lemma from \cite{MMandMS}.

Let $\lambda>0$ be arbitrary. By $\overline{D}_{\lambda}(\zeta)=\{z\in\mathbb{U}:d_{\mathbb{U}}(z,\zeta)\leqslant\lambda\}$ (respectively $\overline{S}_{\lambda}(\zeta)=\{z\in\mathbb{S}:d_{\mathbb{S}}(z,\zeta)\leqslant\lambda\}$) we denote the hyperbolic closed disc in $\mathbb{U}$ (respectively in $\mathbb{S}$) with hyperbolic center $\zeta\in\mathbb{U}$ (respectively $\zeta\in\mathbb{S}$) and hyperbolic radius $\lambda$. Specially, if $\zeta=0$  we omit  $\zeta$ from the notations.


Let $r\in(0,1)$ be arbitrary. By $\overline{U}_r$ we denote Euclidean closed disc
$$\{z\in\mathbb{C}: |z|\leqslant r\}.$$
Also, let
\begin{equation*}
    \lambda(r)=d_{\mathbb{U}}(r,0)=\ln\frac{1+r}{1-r}=2\artanh{r}.
\end{equation*}
Since $\displaystyle d_{\mathbb{U}}(z,0)=\ln\frac{1+|z|}{1-|z|}=2\artanh{|z|}$, for all $z\in\mathbb{U}$, we have
\begin{equation*}
    \overline{D}_{\lambda(r)}=\left\{z\in\mathbb{C}:2\artanh{|z|}\leqslant 2\artanh{r}\right\}=\{z\in\mathbb{C}:|z|\leqslant r\}=\overline{U}_r.
\end{equation*}
Let $b\in(-1,1)$ be arbitrary and $\displaystyle a=\tan{\frac{b\pi}{4}}$. By Theorem \ref{th:schwpick} we have
\begin{equation*}
    \overline{S}_{\lambda(r)}(b)=\overline{S}_{\lambda(r)}(\phi(\varphi_{a}(0)))=\phi(\varphi_{a}( \overline{D}_{\lambda(r)}))=\phi(\varphi_{a}(\overline{U}_r)),
\end{equation*}
where $\phi$ is conformal isomorphism from $\mathbb{U}$ onto $\mathbb{S}$ defined in subsection \ref{subsec:hyperbolic}.
Further, one can show that (see Figure \ref{fig:1}):
\begin{itemize}
  \item[i)] $\overline{S}_{\lambda(r)}(b)$ is symmetric with respect to the $x$-axis;
  \item[ii)] $\overline{S}_{\lambda(r)}(b)$ is Euclidean convex (see \cite[Theorem 7.11]{BeardonMinda}).
\end{itemize}

\begin{figure}[H]
\includegraphics[width=\linewidth]{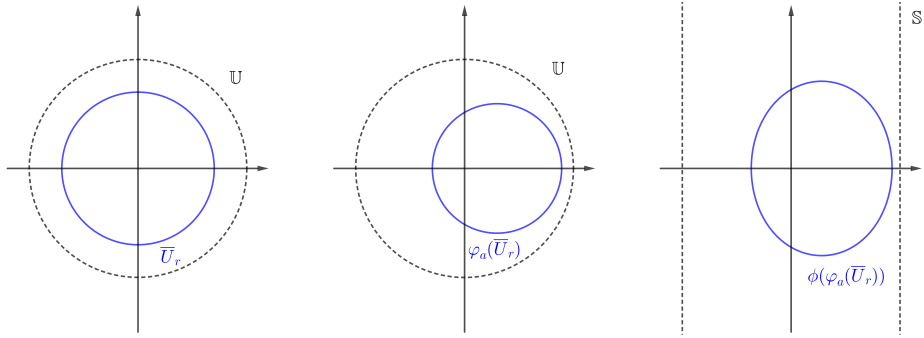}
\caption{Disks $\overline{U}_r$, $\varphi_{a}(\overline{U}_r)$ and $\phi(\varphi_{a}(\overline{U}_r))$}\label{fig:1}
\end{figure}

By i)-ii) it is readable that we have:

\begin{lemma}[{\cite[Lemma 3]{MMandMS}}]\label{lem:hypdisk0}
Let $r\in(0,1)$ and $b\in(-1,1)$ be arbitrary. Then
\begin{equation*}
    R_{e}(\overline{S}_{\lambda(r)}(b))=\left[\frac{4}{\pi}\arctan\varphi_a(-r),\frac{4}{\pi}\arctan\varphi_a(r)\right]
\end{equation*}
Here $\displaystyle a=\tan{\frac{b\pi}{4}}$ and $R_{e}:\mathbb{C}\rightarrow\mathbb{R}$ is defined by $R_{e}{(z)}=\re{z}$.
\end{lemma}

\begin{proof}[Proof of Theorem \ref{th:main1}]
Applying the strip method we obtain that there exists holomorphic function $f:\mathbb{U}\rightarrow\mathbb{S}$ such that $\re{f}=u$, $f(0)=b$ and $|f'(0)|=d$.
Also, we have
\begin{equation*}
|f^h(0)|=\frac{\rho_{\mathbb{S}}(f(0))}{\rho_{\mathbb{U}}(0)}|f'(0)|=\frac{\pi}{4}\frac{1}{\cos{\displaystyle\frac{\pi}{2}b}}d=c.
    \end{equation*}
Let $z\in\mathbb{U}$ be arbitrary. By Theorem \ref{th:BeardonCarne}, taking $\Omega_1=\mathbb{U}$ and $\Omega_2=\mathbb{S}$, we have
\begin{eqnarray*}
  d_{\mathbb{S}}(f(z),b) &\leqslant& \log(\cosh{d_{\mathbb{U}}(z,0)}+|f^h(0)|\sinh{d_{\mathbb{U}}(z,0)})  \\
                            &=&\log\left(\frac{1+|z|^2+2c|z|}{1-|z|^2}\right).
\end{eqnarray*}
Now, we chose a point $R(z)\in[0,1)$ such that
\begin{equation}\label{fla:defRz}
    d_{\mathbb{U}}(R(z),0)=\log\left(\frac{1+|z|^2+2c|z|}{1-|z|^2}\right).
\end{equation}
Note that the equality (\ref{fla:defRz}) is equivalent to the equality
\begin{equation*}
    \frac{1+R(z)}{1-R(z)}=\frac{1+|z|^2+2c|z|}{1-|z|^2}
\end{equation*}
and hence we obtain $\displaystyle R(z)=|z|\frac{c+|z|}{1+c|z|}=|z|\varphi_{c}(|z|)$.
Therefore
\begin{equation*}
    d_{\mathbb{S}}(f(z),b)\leqslant d_{\mathbb{U}}(|z|\varphi_{c}(|z|),0),
\end{equation*}
i.e. $f(z)\in\overline{S}_{\lambda\big(|z|\varphi_{c}(|z|)\big)}(b)$. Finally, by Lemma \ref{lem:hypdisk0}
\begin{equation*}
    u(z)=\re f(z)\in\left[\frac{4}{\pi}\arctan{\varphi_{a}\big(-|z|\varphi_{c}(|z|)\big)}, \frac{4}{\pi}\arctan{\varphi_{a}\big(|z|\varphi_{c}(|z|)\big)}\right].
\end{equation*}

If $z=0$ then it is clear that the inequality (\ref{ineq:main}) is sharp.

In order to prove that inequality (\ref{ineq:main}) is sharp in the case $z\in\mathbb{U}\backslash\{0\}$, we first define the functions $\widehat{\Phi},\widetilde{\Phi}:\mathbb{U}\rightarrow\mathbb{S}$ as follows
\begin{equation*}
    \widehat{\Phi}(\zeta)=\phi\Big(\varphi_{a}\big(-\zeta\cdot\varphi_{c}(\zeta)\big)\Big)
\end{equation*}
and
\begin{equation*}
    \widetilde{\Phi}(\zeta)=\phi\Big(\varphi_{a}\big(\zeta\cdot\varphi_{c}(\zeta)\big)\Big).
\end{equation*}

Let $z\in\mathbb{U}\backslash\{0\}$. Define the functions $\widehat{u}^z,\widetilde{u}^z:\mathbb{U}\rightarrow(-1,1)$ (which depend on $z$) on the following way:
\begin{equation*}
    \widehat{u}^z(\zeta)=\re{\widehat{\Phi}}(e^{-i\arg{z}}\zeta)
\end{equation*}
and
\begin{equation*}
    \widetilde{u}^z(\zeta)=\re{\widetilde{\Phi}}(e^{-i\arg{z}}\zeta).
\end{equation*}
It is easy to check that the functions $\widehat{u}^z$ and $ \widetilde{u}^z$ are harmonic and that they satisfy assumptions \rm{(R1)} and \rm{(R2)}. Also
\begin{equation*}
    \widehat{u}^z(z)=\frac{4}{\pi}\arctan{\varphi_{a}\big(-|z|\varphi_{c}(|z|)\big)}
\end{equation*}
and
\begin{equation*}
    \widetilde{u}^z(z)=\frac{4}{\pi}\arctan{\varphi_{a}\big(|z|\varphi_{c}(|z|)\big)}.
\end{equation*}
\end{proof}

\subsection{Proof of Theorem \ref{th:main2}}

In order to prove Theorem \ref{th:main2}, we need two lemmas.

\begin{lemma}[{\cite[Lemma 1]{colonna}}]\label{lem:colonna}
Let $z,w\in\mathbb{C}$. Then
\begin{equation*}
    \max\limits_{\theta\in\mathbb{R}}|w\cos{\theta}+z\sin{\theta}|=\frac{1}{2}(|w+iz|+|w-iz|).
\end{equation*}
\end{lemma}

\begin{lemma}\label{lem:phiAincreasing}
Fix $z\in\mathbb{U}$. Function $h:(-1,1)\rightarrow\mathbb{R}$ defined by $\displaystyle h(t)=\frac{t+|z|}{1+t|z|}$ is monotonically increasing.
\end{lemma}

\begin{proof}
The proof follows directly from the fact $\displaystyle h'(t)=\frac{1-|z|^2}{(1+t|z|)^2}>0$ for all $t\in(-1,1)$.
\end{proof}

\begin{proof}[Proof of Theorem \ref{th:main2}] Denote by $u$ and $v$ real and imaginary part of $f$, respectively. Let $\theta\in\mathbb{R}$ be arbitrary. It is clear that function $U$ defined by
\begin{equation*}
    U(z)=\cos{\theta}u(z)+\sin{\theta}v(z)
\end{equation*}
is harmonic on the unit disk $\mathbb{U}$, $U(0)=0$ and $|U(z)|\leqslant|f(z)|<1$ for all $z\in\mathbb{U}$.

By Theorem \ref{th:main1} we have
\begin{equation}\label{proofmain2:fla1}
    |U(z)|\leqslant\frac{4}{\pi}\arctan\big(|z|\varphi_{c}(|z|)\big), \quad \mbox{ for all }\quad z\in\mathbb{U},
\end{equation}
where $\displaystyle c=\frac{\pi}{4}|\nabla U(0)|$.

Since
\begin{eqnarray*}
  \nabla U(z) &=& \cos{\theta}\nabla u(z)+\sin{\theta}\nabla v(z) \\
              &=& \cos{\theta}\left(u_x(z)+iu_{y}(z)\right)+\sin{\theta}\left(v_x(z)+iv_{y}(z)\right),
\end{eqnarray*}
by Lemma \ref{lem:colonna} we get
\begin{eqnarray*}
  \max\limits_{\theta\in\mathbb{R}}|\nabla U(z)| &=& \max\limits_{\theta\in\mathbb{R}}|\cos{\theta}\left(u_x(z)+iu_{y}(z)\right)+\sin{\theta}\left(v_x(z)+iv_{y}(z)\right)|  \\
    &=& \frac{1}{2}\big(|u_{x}(z)+iu_{y}(z)+i(v_{x}(z)+iv_{y}(z))| \\
    & & {}+|u_{x}(z)+iu_{y}(z)-i(v_{x}(z)+iv_{y}(z))|\big) \\
    &=& \frac{1}{2}\left(\sqrt{(u_x(z)-v_y(z))^2+(u_y(z)+v_x(z))^2}\right.\\
    & & \left.{}+\sqrt{(u_x(z)+v_y(z))^2+(u_y(z)-v_x(z))^2}\right)\\
    &=& |f_{z}(z)|+|f_{\overline{z}}(z)|\\
    &=&\|df(z)\|.
\end{eqnarray*}
Hence
\begin{equation*}
    |\nabla U(0)|\leqslant\|df(0)\|
\end{equation*}
and
\begin{equation*}
    c=\frac{\pi}{4}|\nabla U(0)|\leqslant\frac{\pi}{4}\|df(0)\|=\frac{\pi}{4}d=C.
\end{equation*}
By Lemma \ref{lem:phiAincreasing}, from (\ref{proofmain2:fla1}) we obtain
\begin{equation}\label{proofmain2:fla3}
    |U(z)|\leqslant\frac{4}{\pi}\arctan\big(|z|\varphi_{C}(|z|)\big), \quad \mbox{ for all }\quad z\in\mathbb{U}.
\end{equation}
Finally, let $z\in\mathbb{U}$ be such that $f(z)\neq0$ and let $\theta$ such that
\begin{equation*}
    \cos{\theta}=\frac{u(z)}{|f(z)|} \quad \mbox{ and } \quad \sin{\theta}=\frac{v(z)}{|f(z)|}.
\end{equation*}
Then $U(z)=|f(z)|$ and hence from (\ref{proofmain2:fla3}) we get the inequality (\ref{main2:ineq}).

If $z\in\mathbb{U}$ be such that $f(z)=0$ then the inequality (\ref{main2:ineq}) is trivial.
\end{proof}

\section{Appendix}

\subsection{Harmonic quasiregular mappings and the Schwarz-Pick type estimates}
Taking into account Remark \ref{rm:problem} we mention some results related to harmonic quasiregular mappings.

Let $D$ and $G$ be domains in $\mathbb{C}$ and $K\geqslant1$. A $C^1$ mapping $f:D\rightarrow G$ we call $K-$quasiregular mapping if
\begin{equation*}
     \|df(z)\|^2\leqslant K|J_{f}(z)|, \quad \mbox{ for all } z\in D.
\end{equation*}
Here $J_{f}$ is Jacobian determinant of $f$. In particular, $K-$quasiconformal mapping is the $K-$quasiregular mapping which is a homeomorphism.

In \cite{MKandMM}, M. Kne\v zevi\'c and M. Mateljevi\'c proved the following result (which can be considered as generalization of Theorem \ref{th:colonna}):
\begin{theorem}
Let $f:\mathbb{U}\rightarrow\mathbb{U}$ be a harmonic $K-$quasiconformal mapping. Then
$$\displaystyle \|df(z)\|\leqslant K\frac{1-|f(z)|^2}{1-|z|^2},\quad \mbox{for all}\quad z\in\mathbb{U}.$$
\end{theorem}

Also, one result of this type was obtained by H. H. Chen \cite{hhChen1}:
\begin{theorem}
Let $f:\mathbb{U}\rightarrow\mathbb{U}$ be a harmonic $K-$quasiconformal mapping. Then
$$\displaystyle \|df(z)\|\leqslant\frac{4}{\pi}K\frac{\cos{\displaystyle\left(|f(z)|\pi/2\right)}}{1-|z|^2},\quad \mbox{for all}\quad z\in\mathbb{U}.$$
\end{theorem}
For further results related to harmonic quasiconformal and hyperbolic harmonic quasiconformal mappings we refer to interested reader to \cite{MMTopics,wan,MMRevue, XChenAFang,MMrckmFil, MKFil, MKMor} and literature cited there.


\textbf{Acknowledgement.}
The author is greatly indebted to Professor M.Mateljevi\'c for introducing in this topic and for many stimulating conversations. Also, the author wishes to express his thanks to Mijan Kne\v zevi\'c for useful comments related to this paper.



\end{document}